\documentclass[12pt,reqno]{amsart}

\usepackage{latexsym}
\usepackage{amsmath}
\usepackage{amscd}
\usepackage{amssymb}
\usepackage{dsfont}

\usepackage{mathrsfs}
\usepackage{comment}
\usepackage{color}
\usepackage{enumerate}
\usepackage{soul}
\usepackage{framed}

\usepackage{graphicx}
\usepackage{subfigure}
\usepackage[colorlinks]{hyperref}
\usepackage{url}

\newtheorem{theorem}{Theorem}[section]

\newtheorem{definition}{Definition}[section]

\newtheorem{assumption}[theorem]{Assumption}

\newtheorem{proposition}[theorem]{Proposition}
\newtheorem{remark}[theorem]{Remark}

\numberwithin{equation}{section}

\begin{document}

\title{The Stochastic Energy-Casimir Method}

\author[Arnaudon, Ganaba, Holm]{Alexis Arnaudon, Nader Ganaba and Darryl Holm}

\address{AA, NG, DH: Department of Mathematics, Imperial College, London SW7 2AZ, UK}

\maketitle

\begin{abstract}
    In this paper, we extend the energy-Casimir stability method for deterministic Lie-Poisson Hamiltonian systems to provide sufficient conditions for the stability in probability of stochastic dynamical systems with symmetries and multiplicative noise. 
    We illustrate this theory with classical examples of coadjoint motion, including the rigid body, the heavy top and the compressible Euler equation in two dimensions.  
    The main result of this extension is that stable deterministic equilibria remain stable in probability up to a certain stopping time which depends on the amplitude of the noise for finite dimensional systems and on the amplitude the spatial derivative of the noise for infinite dimensional systems. 
\end{abstract}


\section{Introduction} 

In 1966, V. I. Arnold's fundamental paper \cite{arnold1966geometrie} showed that ideal fluid mechanics can be cast into a geometric framework. In this framework of differential geometry and Lie group symmetry, the mathematical properties of ideal (nondissipative) classical fluid mechanical systems are easily identified.
For example, Arnold's geometric interpretation in \cite{arnold1966geometrie} of ideal incompressible fluid dynamics as geodesic motion on the group of diffeomorphisms was soon followed by a series of fundamental results in analysis, e.g., in \cite{ebin1970groups}. 
We shall be interested here in another development of Arnold's geometric approach to fluid dynamics, which concerns the nonlinear stability of equilibrium (time-independent) solutions. 
A nonlinear fluid stability method based on the Lyapunov method was already introduced in the early days of geometric mechanics by Arnold in \cite{arnold1965conditions,arnold1966priori} for ideal incompressible fluid flows, whose $L^2$ kinetic energy norm provides the metric for their geodesic interpretation. This approach was extended in \cite{holm1985nonlinear} to the energy-Casimir method, which allows for both kinetic and potential energy contributions and, hence, may be applied to a large class of ideal mechanical systems. This class of systems comprises Hamiltonian systems that admit reduction by Lie group symmetries. Such systems possess Lie-Poisson brackets whose null eigenvectors correspond to variational derivatives of conserved quantities called Casimirs. The Casimirs commute under the Lie-Poisson bracket with any functionals on the symmetry-reduced space, as well as with the system Hamiltonian itself. 
For Hamiltonian systems which do not have a Casimir function, the energy-momentum method, developed in  \cite{simo1991stability}, is used instead. This method uses momentum maps instead of the Casimirs to obtain stability results. 
 
Interest has been growing recently in stochastic perturbations of mechanical systems with symmetries whose dynamics can be investigated in the framework of geometric mechanics. The aim of this new science of stochastic geometric mechanics is to extend to stochastic systems the mathematical understanding gained for deterministic systems by using differential geometry and Lie groups. 
The theory of stochastic canonical Hamiltonian systems began with Bismut \cite{bismut1982mecanique} and was recently updated in geometric terms in \cite{lazaro2008stochastic}. This theory was extended to stochastic ideal fluid dynamics in \cite{holm2015variational} by using Lie group symmetry reduction of a stochastic Hamilton's principle. The general theory was developed and illustrated further for finite dimensional Euler-Poincar\'e variational principles with symmetry, leading to noncanonical stochastic Hamiltonian mechanical systems, in \cite{arnaudon2016noise,cruzeiro2016momentum}.   

The present work will seek sufficient conditions for the probabilistic stability of critical points of stochastic geometric mechanics systems, by using an extension of the energy-Casimir method. 
For this endeavour, we will need to introduce an appropriate notion of stability in probability; so that a stochastic counterpart of the energy-Casimir method can be developed and applied to stochastic dynamical systems. 
The main result of this paper is the proof that a deterministically stable stationary solution remains stable in probability up to a finite stopping time, 
for multiplicative stochastic perturbations which preserve coadjoint orbits. 
This theorem applies only if unique solutions of the stochastic process exist. However, since the stability in probability is valid only for finite time, existence and uniqueness of solutions is only needed locally in time. 

\textit{Plan of the paper.}
Section \ref{preliminaries} reviews the theory of stochastic perturbations of mechanical systems with symmetries developed by \cite{holm2015variational,arnaudon2016noise}. It also distinguishes between the notions of stability in the deterministic and stochastic settings, in the context of the deterministic energy-Casimir method.
Section \ref{energy-Casimir} forms the core of the paper, in which the stochastic energy-Casimir method is developed.  
Section \ref{examples} then illustrates the stochastic modifications of the energy-Casimir stability analysis for several classical examples, including the rigid body, the heavy top, and compressible Euler equations.

\section{Preliminaries}\label{preliminaries}

\subsection{Stochastic mechanical systems with symmetries}
This section begins by defining the type of stochastic perturbations of mechanical systems that we will study in this work. For further detail, we refer the interested reader to\cite{arnaudon2016noise,cruzeiro2016momentum} for finite dimensional systems and \cite{holm2015variational} for infinite dimensional systems. 
Although the theory has been studied quite generally in \cite{cruzeiro2016momentum}, here we will restrict ourselves to the examples in \cite{arnaudon2016noise} and \cite{holm2015variational}. 
Let $G$ be a Lie group and $\mathfrak g$ its Lie algebra. For a probability space $( \Omega, \mathcal F_t, P)$, we consider a Wiener process $W_t$ defined with respect to the standard filtration $\mathcal F_t$. 
The construction is based on the following stochastic Hamilton-Pontryagin variational principle: $\delta S = 0$ for the stochastic action integral,
\begin{align}
    S(\xi,g,\mu)= \int l(\xi)\, dt + \int \left \langle \mu ,\circ\, g^{-1} dg - \xi\, dt + \sum_i \sigma_i \circ dW_t^i\right \rangle\,.
    \label{var-principle}
\end{align}
In this formula, we denote $g\in G$, $\xi \in \mathfrak g$, $\mu \in \mathfrak g^*$, where $\mathfrak g^*$ is the dual of the Lie algebra $\mathfrak g$ under the non-degenerate pairing $\langle \cdot ,\cdot \rangle$. The vector fields $\sigma_i\in \mathfrak g$ represent constant multiples of Lie algebra basis elements and the symbol $\circ$ denotes Stratonovich stochastic integrals. The action integral  \eqref{var-principle} is invariant under left translations of the group $G$.
We refer to \cite{Yoshimura2006209,Yoshimura2007381, bloch2003nonholonomic,holm2008geometric} for more details on Hamilton-Pontryagin principle and \cite{marsden1999book,holm2008geometric}  for the use of Lie groups. 
Upon taking free variations $\delta \xi, \delta \mu$ and $\delta g$, and rearranging terms, we find the momentum map relation $ \frac{\delta l}{\delta \xi} = \mu $ and the Euler-Poincar\'e equation for its stochastic coadjoint motion,
\begin{align}
    d\mu = \mathrm{ad}^*_{(\circ g^{-1} dg)}\mu = \mathrm{ad}^*_\xi \mu\, dt + \mathrm{ad}^*_{\sigma_i} \mu \circ dW_t^i\, ,
    \label{sto-EP}
\end{align}
where the relation $\circ g^{-1} dg = \xi\, dt - \sum_i \sigma_i \circ dW_t^i$ for the left-invariant reduced velocity is imposed by variations with respect to $\mu$, regarded as a Lagrange multiplier. 
Thus, the solutions of the stochastic Euler-Poincar\'e equation \eqref{sto-EP} preserve coadjoint orbits, even in the presence of noise. 

As in many applications, systems depend on extra quantities and in such case the Lagrangian $ l(\xi)$ is replaced by $l(\xi,q)$ to account for the explicit dependence on $q\in V^*$, with $V$ a vector space upon which the Lie group $G$ acts, the stochastic Euler-Poincar\'e equations are then
\begin{align}
\begin{split}
    d \mu &+ \left ( \mathrm{ad}^*_\xi \mu  + r\diamond q\right ) dt  + \sum_i \mathrm{ad}^*_{\sigma_i} \mu   \circ dW_t^i  =0
    \,,\\
    d q & +  q\xi  \, dt  + \sum_i   q\sigma_i \circ dW_t^i=0\,,
    \label{SDE-SM-system}
\end{split}
\end{align}
where $r= \frac{\partial h}{\partial q}$, concatenation $q\xi$ denotes the left Lie algebra action of $\xi\in \mathfrak g $ on $q \in V^*$ and $\diamond: V\times V^*\to \mathfrak g^*$ is defined as $\langle r\diamond q, \xi\rangle = - \langle q\xi, r\rangle$, with appropriate pairings. 
According to \cite{gaybalmaz2016geometric}, the $\sigma_i$ could depend on the advected quantities $q$, as well, but we will restrict ourselves to the case where they are constant.
The two equations in \eqref{SDE-SM-system} can, in fact, be written as a single equation, upon noticing their relation to the coadjoint operator for the semidirect product Lie algebra $\mathfrak g\circledS V$, namely, 
\begin{align}
    d(\mu,q) + \mathrm{ad}^*_{(\xi,r)}( \mu,q) \, dt +\sum_i  \mathrm{ad}^*_{(\sigma_i,0)}( \mu,q)\circ dW_t^i =  0\, .
    \label{SD-compact}
\end{align}
Consequently, stochastic Euler-Poincar\'e systems with advected variables undergo coadjoint motion for semidirect product Lie groups. This means the energy-Casimir stability method will directly transfer to this more general case. 
We refer to \cite{holm1998euler} for more details about this theory in the deterministic case and to \cite{arnaudon2016noise,gaybalmaz2016geometric} for the stochastic derivation. 
These semidirect-product Euler-Poincar\'e equations include the heavy top and the compressible 2D Euler equation, as we will see as examples in section \ref{examples}. 

Provided the Lagrangian $l(\xi)$ in \eqref{var-principle} is hyper-regular, the stochastic Euler-Poincar\'e equation \eqref{SD-compact} can be Legendre transformed into the stochastic Lie-Poisson equation 
\begin{align}
    d f(\mu) = \{f,h\} dt + \sum_i \{ f,\Phi_i\} \circ dW_t^i\, , 
    \label{stocHamSys}
\end{align}
where  $h(\mu) = \langle \mu,\xi \rangle - l(\xi)$ is the standard  Hamiltonian, $\Phi_i (\mu) = \langle \sigma_i ,\mu\rangle$ is called the stochastic potential and $\{\cdot ,\cdot \}$ is the Lie-Poisson bracket for smooth functions $f,h:\mathfrak g^*\to \mathbb R$. 
This means that the stochastic Lie-Poisson system \eqref{stocHamSys} is Hamiltonian in the sense of \cite{bismut1982mecanique}, with stochastic Hamiltonian 
\begin{align}
H(\mu):= h(\mu)\, dt + \sum_i \Phi_i(\mu) \circ dW_t^i
\,.
    \label{stocHamExtens}
\end{align}

\begin{remark}[Energy conservation]
	
Notice that the stochastic extension in \eqref{stocHamExtens} breaks the time-translation invariance of the original deterministic Hamiltonian equation. Consequently, it breaks the energy conservation assumed in applying the deterministic energy-Casimir method.  

\end{remark}
\subsection{Deterministic stability}

Before studying the stochastic stability of the general class of stochastic systems described above, let us review the deterministic notions of stability needed here.
Consider a generic dynamical system $\dot x = f(x)$ with a fixed point $x_e$, i.e. $f(x_e)= 0$, where $f: \mathbb R^n\to \mathbb R^n$ for $x \in \mathbb R^n$.  
One of the strongest notions of stability in this situation is called nonlinear stability. 

\begin{definition}[Nonlinear stability]
    An equilibrium position $x_e$ of a dynamical system is nonlinearly stable, if for every neighbourhood $U$ of $x_e$ there exists a neighbourhood $V$ such that trajectories starting in $V$ remain in $U$. 
    In mathematical terms, given an appropriate norm $\|\cdot \|$, $\forall \epsilon>0 $, $\exists$ $\delta >0$ such that if $\|x(0)- x_e\| < \delta $, then $\| x(t) - x_e\| <\epsilon$ for $t>0$. 
\end{definition}

In practice, for finite dimensional systems, there is an equivalent formulation for nonlinear stability, which is called formal stability. 

\begin{definition}[Formal stability]
    An equilibrium point $x_0$ is formally stable, if a conserved quantity exists with vanishing first variation and definite second variation when evaluated at the equilibrium position. 
\end{definition}

In infinite dimensions, such equivalence is not valid, thus nonlinear stability requires stronger conditions on the system and its solution.  
In this case, nonlinear stability may be derived by imposing an additional convexity condition, see \cite{holm1985nonlinear} for more details.  


\subsection{Stochastic stability}

We now describe the corresponding notion of stability for stochastic systems.
Consider a generic stochastic differential equation
\begin{align} 
    dx_t = f(x_t) \, dt + \sum_i \sigma_i(x_t) \circ dW_t^i\, ,
    \label{stoch_system_generic}
\end{align}
where we take $x_t:= x(t)$ and  $\sigma_i:\mathbb R^n\to \mathbb R^n$ are given vector fields. 
Certain conditions on $f$ and $\sigma_i$ must be imposed for the existence of solutions of this stochastic process, which we shall assume are satisfied \cite{kloeden1992numerical}.
In our case, the noise amplitudes $\sigma_i(x)$ will always be multiplicative, thus fixed points can be found such that $dx_e=0$ or $f(x_e) = \sigma_i(x_e)= 0 $, $\forall i$. 
The existence of such fixed points will impose certain restrictions on the possible choice of vector fields $\sigma_i$.
We will need to extend the standard notion of stability to the stochastic case by including a stopping time. 

\begin{definition}[Transient stochastic stability]
    For a suitable norm $\|\cdot \|$, a fixed point $x_e$ is (weakly) stable in probability with stopping time $T$ (or transient stability) if for every $\epsilon,\sigma > 0 $ there exists $\delta >0$ and $T(\epsilon,\delta,C)>0$ such that if $\|x(0) - x_e \|< \delta $ and  $ T> t > 0$ 
    \begin{align}
        P \left( \| x(t) - x_e  \| > \epsilon \right) <   \sigma\, ,
    \end{align}
    where $P(A)$ is the probability of the event $A$ on the same probability space as the original stochastic process. 
\end{definition}

We will refer to this notion as stochastic stability throughout this work. 
Notice that in the case when $T\to \infty$ this notion coincides with the usual notion of stochastic stability. 
We refer to \cite{khasminskii2011stochastic} for more detailed discussions about the theory of stochastic stability.     

The interpretation of the notion of stochastic stability with stopping time needs some clarification and should not be misinterpreted. 
First, the stability is only in probability, even if there is no stopping time. 
This means that for large enough times the system will behave as if it is unstable, that is it will leave the region near the stationary position. 
This cannot be prevented as we assume, according H\"ormander theorem that these systems are ergodic, thus a single path will eventually visit the entire phase space. 
This notion of stability in probability is thus a useful notion only for studying these stochastic systems for short times.
In the long time regime, one needs to study the invariant measure of the associated Fokker-Planck equation or other objects such as random attractors to understand the long time behaviour of these systems, see \cite{arnaudon2016noise}. 
The present work is thus only devoted to the short time analysis of the same systems. 

\subsection{The energy-Casimir method}

We now describe the deterministic energy-Casimir method following the exposition of \cite{holm1985nonlinear} and introduce important notations for the next section \ref{energy-Casimir}.
We assume that our mechanical system can be written in Lie-Poisson form with Hamiltonian $h:\mathfrak g^*\to \mathbb R $ and Casimir function $C:\mathfrak g^*\to \mathbb R$. 
Recall that the Casimir function is constant on the coadjoint orbits and corresponds to a constant of motion of the Lie-Poisson Hamiltonian system. 
We assume that this system has an equilibrium position $\mu_e:\mathfrak g^* $, that is $\dot \mu_e= 0$.

To simplify the formulas, we introduce abbreviated notations for the variations of an arbitrary function $F:\mathfrak g^* \to \mathbb R$. 
We will write the first variation of $F$ as a map $DF:\mathfrak g^*\to \mathfrak g$ which is understood as the variational derivative of $F$ with respect to its argument $\mu$. 
The second variation of $F$ is the variational derivative of $DF(\mu)$ with respect to its argument and is a map $D^2F: \mathfrak g^*\to \mathfrak g\times \mathfrak g$. 
When evaluated at a point $\mu$ this map is $D^2F(\mu):\mathfrak g^*\times \mathfrak g^*\to \mathbb R$ provided $\mathfrak g^{**}\simeq \mathfrak g$. 
Equivalently, we can write $D^2F(\mu):\mathfrak g^*\to \mathfrak g$ for any $\delta \mu_1,\delta\mu_2\in \mathfrak g^*$, as
\begin{align*}
    D^2F(\mu)(\delta \mu_1, \delta \mu_2) = \langle \delta \mu_1, D^2F(\mu) \delta\mu_2\rangle_{\mathfrak g^*\times \mathfrak g}\, . 
\end{align*}
We now simply describe the energy-Casimir method; some proofs will be given in the stochastic context below. 
We begin by defining the extended Hamiltonian $H_c:= h+C$.  
We then need to find conditions on the form of the Casimir $C$ such that $\langle DH_C (\mu_e) ,\delta \mu\rangle= 0 $ for all $\delta \mu\in \mathfrak g^*$.  
More explicitly, we have to find $\lambda$ such that 
\begin{align*}
    \langle D H_C(\mu_e),\delta \mu\rangle  = \langle D h(\mu_e),\delta \mu\rangle  + \lambda \langle D C(\mu_e),\delta \mu\rangle = 0\, . 
\end{align*}
The last step is to compute the quadratic form $\langle \delta \mu, D^2H_C(\mu_e) \delta \mu\rangle $ and find additional conditions on the Casimir as well as specific relations between the equilibrium position $\mu_e$ and certain parameters of the system for this quantity to be sign-definite. 
These relations will give the conditions for nonlinear stability. 
Indeed, if these can be found, then the quadratic form
\begin{align}
    \| \delta \mu \|_{H_C}^2:= \pm \langle \delta \mu, D^2H_C(\mu_e) \delta \mu \rangle\, 
    \label{HC-norm}
\end{align}
comprises a norm on the space of variations $\delta \mu \in \mathfrak g^*$.
The choice of sign is to indicate the quadratic form is either negative definite, or positive definite. 
One may directly check that this norm is preserved by the linearization of the deterministic flow. 
In finite dimensions, this property of the norm corresponds to the definition of nonlinear stability stated before. 
For infinite dimensions, one need stronger conditions obtainable from convexity estimates leading to both lower and upper bounds, in order to claim nonlinear stability. 

\section{The stochastic Energy-Casimir method}\label{energy-Casimir}

We are now ready to extend the energy-Casimir method to the stochastic systems described in the previous section. 
We begin by describing the finite dimensional energy-Casimir method in the Euler-Poincar\'e setting, without advected quantities. 
We will treat the latter at the end of this section. 
First, suppose there exists an equilibrium solution $\mu_e$ of the deterministic dynamical system corresponding to \eqref{sto-EP}.
This condition depends on the geometry and the Hamiltonian of the system, which is written explicitly as 
\begin{align*}
    \mathrm{ad}^*_{Dh(\mu_e)} \mu_e = 0\,.  
\end{align*}
For such a fixed point $\mu_e$ to persist in the stochastically perturbed equation we need a particular choice for the noise vector fields $\sigma_k$, that is 
\begin{align*}
    \mathrm{ad}^*_{\sigma_k}\mu_e = 0\, , \quad \forall k\, . 
\end{align*}
As we will see in the examples, there exist particular choices of the noise fields such that this condition is satisfied at least for the rigid body, heavy top and compressible fluids. 
In particular, this condition holds for the rigid body when the noise is aligned with the direction of the fixed point on the angular momentum sphere. 

The next step is the same as in the deterministic case. Namely, one seeks conditions on the Casimir such that the first variation of the extended Hamiltonian $H_C:= h+ C$ vanishes at the equilibrium solution $\mu_e$, i.e. $DH_C(\mu_e)=0$.
The third step is to check the sign definiteness of the second variation of $H_C$ evaluated at $\mu_e$. 
This does not depend on the noise but on the form of $H_C$ and the geometry of the system. 
We can thus use the results from the deterministic case and assume from now on that the point $\mu_e$ is deterministically stable, that is $D^2H_C(\mu_e)$ is sign-definite and corresponds to the norm \eqref{HC-norm}.
From now on, we will use the convention $D^2H_C:= D^2H_C(\mu_e)$ to simplify the notation. 
The difference with the deterministic case enters in the previous step of verifying that the linearized flow conserves the linearized Hamiltonian. 
This conservation will not hold when the stochastic perturbations are present, because of their time dependence.  
We first compute the linearised flow using $\mu= \mu_e + \epsilon \delta \mu$ for $\epsilon \ll 1$ to obtain the stochastic process for $\delta \mu$
\begin{align}
    d\delta \mu_t  = \mathrm{ad}^*_{D h(\mu_e)} \delta \mu_t\, dt 
    + \mathrm{ad}^*_{D^2h(\mu_e)\delta \mu_t} \mu_e\, dt + \sum_i \mathrm{ad}^*_{\sigma_i} \delta \mu_t \circ dW_t^i\, ,
\end{align}
which, when we use the extended Hamiltonian $H_C$ instead of $h$ becomes
\begin{align}
    d\delta \mu_t  =  \mathrm{ad}^*_{D^2H_C \delta \mu_t} \mu_e\, dt + \sum_i \mathrm{ad}^*_{\sigma_i} \delta \mu_t \circ dW_t^i\, .
\end{align}
It is now clear that in general this flow will not preserve the $H_C$ norm \eqref{HC-norm}.
Therefore, we need to estimate by how much this quantity is not preserved by the flow. 
For this, we compute the time evolution of $H_C$ norm of $\delta \mu$ to find  
\begin{align*}
    d\|\delta \mu_t\|^2_{H_C} &=  2 \langle  D^2H_C \delta\mu_t,d\delta \mu_t\rangle = \sum_i 2 \langle  D^2H_C \delta\mu_t,  \mathrm{ad}^*_{\sigma_i} \delta \mu_t \rangle\circ dW_t^i\, ,
\end{align*}
where we used the symmetry of $D^2H_C$ under the pairing $\langle \cdot, \cdot \rangle$.  
Recall now that the notion of stochastic stability is only in probability. 
  For dealing with this we need to introduce a specific type of stochastic derivative for a stochastic process $y_t$
\begin{align} 
    \mathcal D_y f(y_t):= \lim _{\Delta t\to 0 } \mathbb E \left (\left . \frac{f(y_{t+\Delta t})- f(y_t)}{\Delta t}\right | \mathcal F_t\right )\, ,
    \label{D-def}
\end{align}
which is also called the infinitesimal generator of the process $y_t$ for any function $f: \mathfrak g^* \to \mathbb{R}$, see for example \cite{oksendal2003stochastic}. 
We use the notation $\mathcal D_y$ as this operation does not contain a partial time derivative, as would have occurred if the definition \eqref{D-def} were applied on a time-dependent function $f(y)$. 
Also, $\mathcal D_y f(y_t)$ is also known as the backwards Kolmogorov operator and the Fokker-Planck equation for the process $y$ is nothing else than 
\begin{align}
    \partial_t \mathbb P(y,t) = \mathcal D_y^*\mathbb P( y,t)\, ,
\end{align}
where $\mathcal D_y^*$ is the formal adjoint of $\mathcal D_y$, in the $L^2$ pairing. 
We refer to \cite{oksendal2003stochastic} for more details about this operator in the context of mechanics. 
Using this derivative now allows us to compute 
\begin{align}  
    \mathcal D_{\delta \mu} \|\delta \mu_t\|_{H_C}^2 
    &=  \sum_i \langle (D^2H_C\mathrm{ad}^*_{\sigma_i} \delta \mu_t  +  \mathrm{ad}_{\sigma_i} D^2H_C \delta\mu_t )\,,\, \mathrm{ad}^*_{\sigma_i} \delta \mu_t\rangle\, .
    \label{dtildeH}
\end{align}

\begin{remark}[Bi-invariant pairings]
Notice that for bi-invariant pairings this formula \eqref{dtildeH} simplifies to 
\begin{align}
    \mathcal D_{\delta \mu} \|\delta \mu_t\|_{H_C}^2 &=  \sum_i \langle   D^2H_C\mathrm{ad}_{\sigma_i} \delta \mu_t -\mathrm{ad}_{\sigma_i} D^2H_C \delta\mu_t , \mathrm{ad}_{\sigma_i} \delta \mu_t\rangle\, ,
    \label{dtildeH-SS}
\end{align}
which shows that this term consists of the commutation between $\mathrm{ad}_\sigma$ and the Hessian of the extended Hamiltonian $H_C$. 
\end{remark}

We now do the following assumption. 
\begin{assumption}\label{assumption}
We assume that we can bound \eqref{dtildeH} as  
\begin{align}
    \mathcal D_{\delta\mu} \|\delta \mu_t\|^2_{H_C} \leq \Sigma^2 \|\delta \mu_t\|_{H_C}^2\, , 
    \label{DtH}
\end{align}
for a given constant $\Sigma\in \mathbb R$.
\end{assumption}

Finding a good estimate for this constant can only be done case by case, and will provide important information on the stability of the fixed point. 
Notice that if the noise vanishes, i.e. $\Sigma=0$, then the norm $\|\delta \mu\|^2_{H_C}$ is conserved. 
The estimate \eqref{DtH} provides a bound of Gr\"onwall type on the time evolution of the expected value of the linearized Hamiltonian, as described in the next proposition.

\begin{proposition}
    Assuming that \eqref{DtH} holds for some $\Sigma$, we have the following time dependent bound for the expected value of the $H_C$ norm of $\delta \mu_t$ 
\begin{align}
    \mathbb{E}(\|\delta \mu_t\|_{H_C}^2)\leq \|\delta \mu_0\|_{H_C}^2 e^{\Sigma^2 t}\, . 
    \label{Htilde}
\end{align}
\end{proposition}

\begin{proof}
    We compute
    \begin{align*}
        \frac{d}{dt} \mathbb E(\|\delta \mu_t\|^2_{H_C}) &= \frac{d}{dt}\int \|\delta \xi\|^2_{H_C} \mathbb P(\xi, t)\, d\xi\\
        &= \int\|\xi \|^2_{H_C} \frac{d}{dt}\mathbb P(\xi, t)\, d\xi\\
        &= \int\|\xi \|^2_{H_C} D_\xi^* \mathbb P(\xi, t)\, d\xi\\
        &= \int \mathcal D_\xi \|\xi \|^2_{H_C} \mathbb P(\xi, t)\, d\xi\\
        &\leq \Sigma^2  \int \|\xi\|^2_{H_C} \mathbb P(\xi, t)\, d\xi\\
        &= \Sigma^2\,  \mathbb E( \|\delta \mu_t\|^2_{H_C} )\, , 
    \end{align*}
     where $\mathbb P$ is the probability density of the process $\delta \mu$ and $\mathcal D_\xi$ its generator \eqref{dtildeH}, which is symmetric in this case as the drift term vanishes.
    Finally, a direct time integration gives \eqref{Htilde}.
    Alternatively, we can use Dynkin's formula and with that we have
        \begin{align*}
      \mathbb E(\|\delta \mu_t\|^2_{H_C}) &= \|\delta \mu_0\|^2_{H_C} +  \int \mathcal D_\xi \|\xi \|^2_{H_C} \mathbb P(\xi, t)\, d\xi \, dt \\
        &\leq  \|\delta \mu_0\|^2_{H_C} + \Sigma^2  \int \|\xi\|^2_{H_C} \mathbb P(\xi, t)\, d\xi \, dt\\
        &= \|\delta \mu_0\|^2_{H_C} + \Sigma^2  \int \mathbb E(\|\delta \mu_t\|^2_{H_C})  \, dt \, , 
    \end{align*}
     then using Picard sequence of successive approximations, we obtain \eqref{Htilde}. The main advantage of using Dynkin's formula is that we do not have to make any assumption on the sign of the right hand side of the inequality, which makes it suitable when considering systems with geometric dissipation. 
\end{proof}

We can now prove the main result of this paper. 
\begin{theorem}\label{SEC}
    Let $\mu_e$ be an equilibrium state of the stochastic Euler-Poincar\'e equation \eqref{sto-EP}. 
    Then $\mu_e$ is transiently stable in probability if it is nonlinearly stable for the deterministic dynamics. 
\end{theorem}
\begin{proof}
    Given the constant $\Sigma$ estimated before, we apply the standard Markov inequality to the norm $\|\delta \mu\|_{H_C}^2$ and obtain
    \begin{align*}
        P(\|\delta \mu_t\|_{H_C}^2>\epsilon ) &\leq \frac{1}{\epsilon} \mathbb E( \|\delta \mu_t\|_{H_C}^2) \\
         &\leq \frac{1}{\epsilon}  \|\delta \mu(0)\|_{H_C}^2 e^{\Sigma^2 t}\, . 
    \end{align*}
    Thus, for all $\epsilon>0 $ we can find $\delta= \|\delta \mu(0)\|_{H_C}^2$ and $ \sigma \geq \frac{1}{\epsilon}  \|\delta \mu(0)\|_{H_C}^2 e^{\Sigma^2 t}$  and the stopping time $T_\mathrm{max} = \frac{1}{\Sigma^2}\mathrm{ln}\left (\frac{\epsilon}{\delta} \right )$. 
    For infinite dimensions, the nonlinear stability condition is obtained via the introduction of a stronger norm which satisfies certain convexity estimates, see \cite{holm1985nonlinear}. 
    This norm is controlled by the extended Hamiltonian for all time.
    The same argument thus applies and the stochastic nonlinear stability condition holds in infinite dimensions. 
\end{proof}

The proof of this theorem demonstrates the importance of the constant $\Sigma$ which gives the rate of expansion of the subspace of the phase space which contains the expected trajectory of the system. 

\begin{remark}[Semidirect products]
The analysis for the semidirect-product case in \eqref{SD-compact} follows analogously, and it will not be written out here. The only difference for the semidirect-product case lies in the estimation of the constant $\Sigma$. 
We will estimate $\Sigma$ case by case in the example section \ref{examples}.
\end{remark}

\section{Examples}\label{examples}

This section applies the stochastic Hamiltonian stability theory to a few illustrative examples. 
We treat the rigid body in section \ref{RB}, the heavy top in section \ref{HT} and the compressible Euler equation in section \ref{Euler-compressible}. 
The incompressible fluid case will be obtained as a consequence of the result of the compressible case. 

\subsection{The stochastic free rigid body}\label{RB}

We will not review the derivation of the rigid body from physical principles, such derivation can be found in \cite{marsden1999book} and \cite{arnaudon2016noise} for its stochastic deformation. 
In short, the Lie algebra is $\mathfrak g= \mathfrak{so}(3)$, which we identify with $\mathbb R^3$. 
In this case, the cross product is the Lie bracket and the negative coadjoint action. 
The reduced momentum is denoted $\boldsymbol \Pi$ and the equation for the stochastic rigid body is
\begin{align}
    d\boldsymbol\Pi + \boldsymbol\Pi\times \boldsymbol\Omega\, dt + \sum_i\boldsymbol\Pi\times \boldsymbol{\sigma}_i \circ dW^i_t=0\, ,
    \label{Sto-RB-stra}
\end{align}
where the Hamiltonian is $h(\Pi) = \frac12 \boldsymbol \Pi \cdot \mathbb I^{-1} \boldsymbol \Pi:=\frac12 \boldsymbol \Pi \cdot \boldsymbol \Omega$, with moment of inertia $\mathbb I= \mathrm{diag}(\mathbb I_1,\mathbb I_2,\mathbb I_3)$. 

The theory for stochastic stability now follows by assuming that the position $\Pi_1$ is stable. 
This is the case if we choose $\mathbb I_1> \mathbb I_2> \mathbb I_3$. 
The only noise compatible with this equilibrium position consists of a single Wiener process with amplitude $\boldsymbol \sigma= \sigma e_1$. 
The Casimir for this system is $C(\boldsymbol \Pi) = \Phi(\frac12 \|\boldsymbol \Pi \|^2)$ for an arbitrary function $\Phi:\mathbb R\to \mathbb R$ and the extended Hamiltonian is 
\begin{align}
    H_C(\Pi) = \frac12 \boldsymbol \Pi\cdot \boldsymbol \Omega + \Phi\left (\frac12 \|\boldsymbol \Pi\|^2\right )\, .
\end{align}
The deterministic condition for stability is that $\mathbb I_1$ is either a major or minor axis. 
In the other two cases there is a function $\Phi$ which gives the positive definiteness of $\|\delta \boldsymbol \Pi \|^2_{H_C}$. That is, the quantity
\begin{align*}
    D^2H_C(\Pi_e) (\delta \boldsymbol \Pi)^2= \left (\frac{1}{\mathbb I_2}- \frac{1}{\mathbb I_1}\right )\delta \Pi_2^2 + \left (\frac{1}{\mathbb I_3}- \frac{1}{\mathbb I_1}\right )\delta \Pi_3^2 + \Phi''\left (\frac12\right ) \delta \Pi_1^2
\end{align*}
is positive definite. 
By directly computing \eqref{dtildeH-SS} we obtain the following for the estimation of the constant $\Sigma$  
\begin{align*}
    \mathcal D_{\delta \Pi}\|\delta \boldsymbol \Pi \|^2_{H_C} &=  \sigma^2\left (\delta \Pi_3^2 (\mathbb I_3^{-1}-\mathbb I_2^{-1}) + \delta \Pi_2^2  (\mathbb I_2^{-1}-\mathbb I_3^{-1})\right ) \\
    &\leq \sigma^2\left (  \delta \Pi_3^2 ( \mathbb I_3^{-1}-\mathbb I_1^{-1}) + \delta \Pi_2^2 (\mathbb I_2^{-1}- \mathbb I_1^{-1})\right ) \\
    &\leq \sigma^2\|\delta \boldsymbol \Pi \|^2_{H_C}\, ,
\end{align*}
where the last inequality is valid in the positive definite case, and the negative definite case just needs an overall minus sign. 
Thus $\Sigma^2 = \sigma^2>0$ is directly proportional to the noise amplitude. 
We can then apply theorem \ref{SEC} to obtain transient stochastic stability of this equilibrium. 

\begin{remark}[Kubo oscillator]
	
If $\mathbb I_2= \mathbb I_3$, the stability result is stronger; since $\mathbb I_2= \mathbb I_3$ implies $\Sigma=0$. 
In this case, the system is equivalent to the Kubo oscillator whose dynamics takes place on circles around the $e_1$ direction, see \cite{arnaudon2016noise}. 

\end{remark}
\begin{figure}[htpb]
    \centering
    \subfigure[Stochastic rigid body]{\includegraphics[scale=0.45]{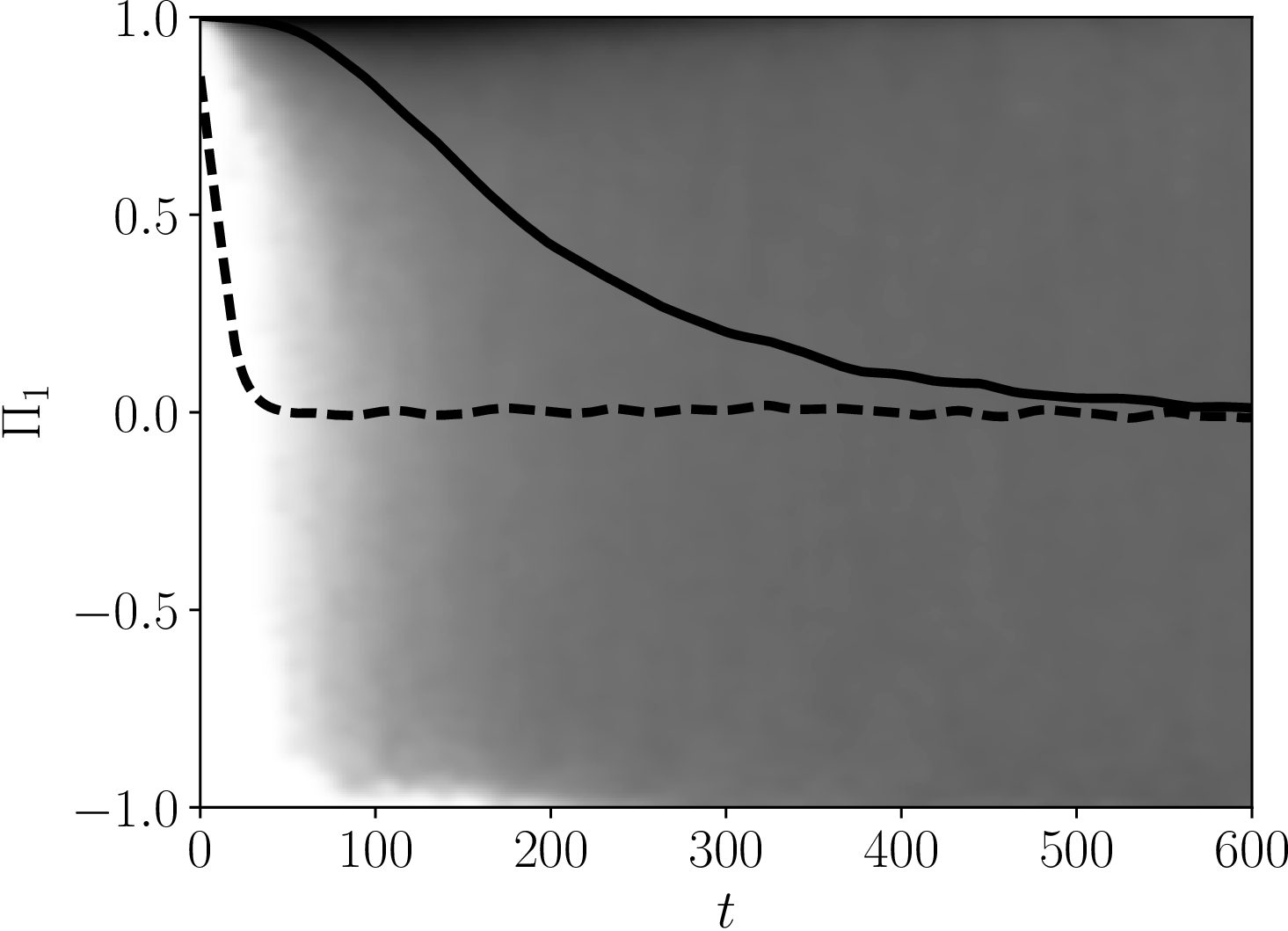} \label{fig:RB-stab}}
    \subfigure[Stochastic heavy top]{\includegraphics[scale=0.45]{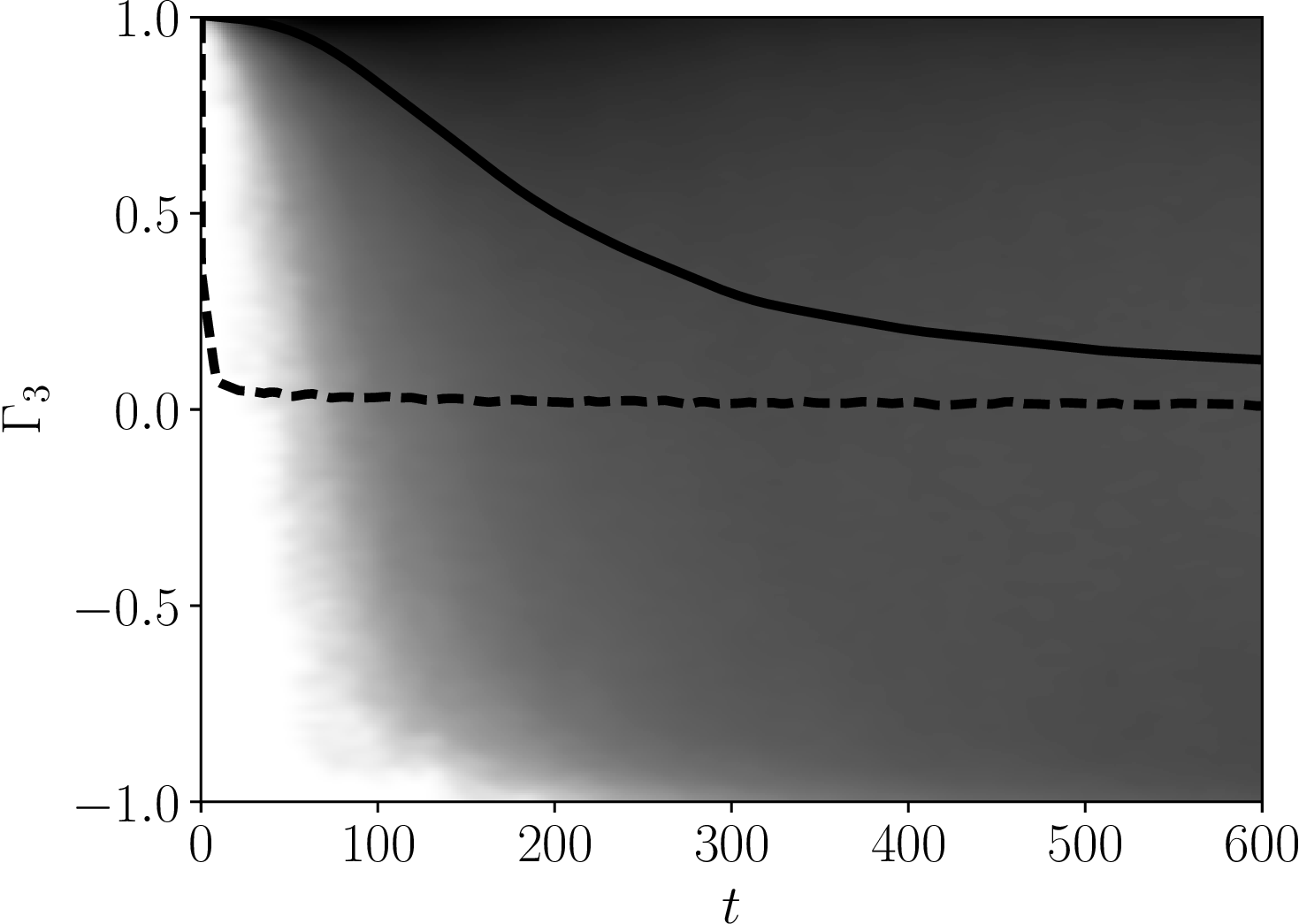} \label{fig:HT-stab}}
    	\caption{ This figure illustrates the stochastic stability of the rigid body (left) and heavy top (right).   The black solid line is the mean path and the black dashed line is the mean of the the same simulation but with an isotropic noise, $\sigma_i=\sigma e_i$.  The gray shades in the background represent the probability density of the solution  $\Pi_1$  and $\Gamma_3$ for the rigid body and the heavy top, respectively.   The transition between deterministic stable regimes occurs near $\Pi_1=1$ and $\Gamma_3 = 1$, respectively, and a uniform distribution illustrates  the stopping time for this transition.  }
\end{figure}

We illustrate the effect of losing stability for large times via a simple numerical experiment, in which we sampled many stochastic rigid bodies with the same initial condition close to the equilibrium $\boldsymbol \Pi=(1,0,0)$. 
We display the result in figure \ref{fig:RB-stab} with the mean trajectory of the projection of the momentum on the $e_1$ axis of the body. 
On figure \ref{fig:RB-stab}, one observes a transition between a regime of stability in probability and an ergodic regime.
The timescale for this transition is given by the stopping time of the stability theory. 
Notice that the stopping time is only an upper bound, thus we are guaranteed to lose stability for larger times. 

\subsection{The stochastic heavy top}\label{HT}

We will follow the example of \cite{holm1985nonlinear} for the semidirect product and consider the integrable heavy top called the Lagrange top, which is the case when $\mathbb I_1= \mathbb I_2$. 
The advected variable is $\boldsymbol \Gamma \in \mathbb R^3$ and the Hamiltonian is $h(\boldsymbol \Pi, \boldsymbol \Gamma) = \frac12 \boldsymbol \Pi\cdot \boldsymbol \Omega + mgl \boldsymbol \chi \cdot \boldsymbol \Gamma$, where $\boldsymbol \chi= e_3$. 
We refer to \cite{lewis1992heavy} for a complete study of the stability of the heavy top in the general case. 
The stochastic heavy top equations follow directly from \eqref{SDE-SM-system} and  read
\begin{align}
    \begin{split}
    d\boldsymbol \Pi &+(\boldsymbol \Omega\, dt + \sum_i \boldsymbol{\sigma}_i\circ dW^i_t) \times\boldsymbol \Pi +mgl (\boldsymbol \Gamma \times \boldsymbol \chi)dt =0 \,,\\
d\boldsymbol \Gamma &+(\boldsymbol \Omega\, dt + \sum_i\boldsymbol{\sigma}_i\circ dW^i_t)\times\boldsymbol \Gamma= 0 \,.
    \end{split}
\end{align}
The deterministic equilibrium position is $\Pi_e= (0,0,\Pi_3)$ and $\Gamma_e= (0,0,1)$, and the stability criterion found from the deterministic energy-Casimir method is $\Pi_3^2> 4mg l \mathbb I_1$. 
The compatible noise is $\boldsymbol \sigma = \sigma e_3$. 
Notice that this stochastic system is integrable, see \cite{arnaudon2016noise}.
As before, we estimate the constant $\Sigma$ by a direct computation of \eqref{dtildeH}. 
For this, we need the adjoint and coadjoint actions
\begin{align*}
    \mathrm{ad}_{(\Pi,\Gamma)} (\Pi',\Gamma')&= \left (\Pi\times \Pi', \Pi\times \Gamma' + \Gamma\times \Pi'\right )\\
    \mathrm{ad}^*_{(\Pi,\Gamma)} (\Pi',\Gamma') &= -\left ( \Pi\times \Pi' + \Gamma\times \Gamma', \Pi\times \Gamma'\right )\, . 
\end{align*}
By substituting the form of $\boldsymbol \sigma$, these equations simplify to
\begin{align*}
    \mathrm{ad}^*_{(\sigma e_3,0)} (\delta \Pi, \delta \Gamma)&= -\sigma ( e_3\times \delta \Pi, e_3\times \delta \Gamma)\\ 
    \mathrm{ad}_{(\sigma e_3,0)} (\xi, \eta)&= \sigma ( e_3\times \xi, e_3\times \eta)\, .
\end{align*}
Remarkably, the coadjoint and adjoint action only differ by a sign, as also occurs for semi-simple Lie algebras.
This implies that we only need to compute a simplified version of the formula \eqref{dtildeH}, that is 
\begin{align*}
    \mathcal D_{(\delta \Pi_t,\delta \Gamma_t)}\|(\delta \Pi_t,\delta \Gamma_t) \|_{H_C}^2 &= \sigma^2\langle -D^2H_C\mathrm{ad}_{(e_3, 0 ) } (\delta \Pi_t,\delta \Gamma_t)   \\
    &\hspace{20mm}+  \mathrm{ad}_{(e_3,0) } D^2H_C (\delta \Pi_t,\delta \Gamma_t),  ( e_3\times \Pi, e_3\times \Gamma)\rangle \\
    &= \sigma^2(\delta \Pi_3^2 -\delta \Pi_2^2)\left ( \frac{1}{\mathbb I_3}-\frac{1}{\mathbb I_2} \right ) +\sigma^2 (\delta \Gamma_3^2- \delta \Gamma_2^2)  \\
    &\leq  \sigma^2\delta \Pi_3^2  \frac{1}{\mathbb I_3}  + \sigma^2\delta \Pi_2^2\frac{1}{\mathbb I_2}  +\sigma^2 (\delta \Gamma_3^2+ \delta \Gamma_2^2)\\
    &\leq \sigma^2 \|(\delta \Pi_t,\delta \Gamma_t)\|^2_{H_C}\, . 
\end{align*}
For this computation we combined the second variation of the extended Hamiltonian with the condition on the Casimir which provides the nonlinear stability results. 
The complete equation is rather lengthy, so we just refer to \cite{holm1985nonlinear} for more details.
The important point is that the last step in this computation uses the positivity of the quadratic form to bound the expected time evolution of the norm of the variations by the norm itself. 
Then, as in the case of the rigid body, we have that $\Sigma^2 = \sigma^2$.
We can finally apply the theorem \ref{SEC} to obtain transient stochastic stability of this equilibrium. 

Similar to the rigid body example, we rely on numerical tests to demonstrate the transient stability of the stochastic heavy top. Figure  \ref{fig:HT-stab} shows the  mean of the sample paths of $\Gamma_3$ and their probability density function starting from the same initial condition, which is close to the equilibrium solution  $\boldsymbol \Gamma = (0,0,1)$. As expected, $\Gamma_3$ exhibits a transition from a stable position to an ergodic position. 

We will not further investigate the stability of the stochastic heavy top here, but only remark that $\Sigma^2$ does not depend for example on the gravity, but only on the noise amplitude. 
We expect this to hold for the more general heavy top with three different moments of inertia, as the restriction on the noise amplitude will remain the same, and only the deterministic stability criterion will be different, see \cite{lewis1992heavy}.

\subsection{The stochastic compressible Euler equation}\label{Euler-compressible}

In the previous section, we studied the classical finite dimensional examples and we saw that the constant $\Sigma$ is exactly the noise amplitude $\sigma$. 
We will now study the classical infinite dimensional examples which are the compressible and incompressible two-dimensional Euler equations. 
We start with the stochastic compressible Euler equation which is more general but directly imply the result for the incompressible system. 
Here we will just apply the previous theory, but we refer to \cite{holm2015variational,gaybalmaz2016geometric} for more details about these stochastic fluid equations and how to derive them.  For establishing formal stability of deterministic compressible Euler equation, we will only consider barotropic flow, i.e. the fluid's density depends on the pressure only. For the result to hold, we need to make the following assumptions: shocks and cavities do not develop and the flow remains in the subsonic regime.  

The compressible fluid is a semi-direct product system with the density $\rho$ is
an advected variable, and $u\in \mathfrak X(\mathbb R^2)$ is a velocity field.  
The equation of motion is a coadjoint motion on the semi-direct product between the diffeomorphism group and the space of densities and the corresponding adjoint and coadjoint actions are 
\begin{align}
    \mathrm{ad}_{(u,\rho)} (v,\rho')&=( (u\cdot \nabla) v-(v\cdot \nabla)u , u\cdot \nabla \rho'- v\cdot \nabla \rho)\nonumber \\
    &= \sum_{j}(u_j\partial_j v_i- v_j \partial_j u_i, u_i \partial_i\rho' - v_i \partial_i \rho)\\ 
    \mathrm{ad}^*_{(u,\rho)} (m,\eta)&= - ( (u\cdot \nabla )m + m\cdot (\nabla u)^T + m(\nabla \cdot u)+ \eta \nabla \rho , \nabla\cdot ( \eta u))\nonumber \\
    &= - \sum_{j}( m_j \partial_i u_j + \partial_j m_i u_j - m_i \partial_j u_j + \eta \partial_i \rho, \partial_j(\eta u))\, .
\end{align}
Meanwhile, the deterministic Hamiltonian is  
\begin{align}
    h(u, \rho) = \int \left ( \frac12 \rho u^2 + \varepsilon(\rho) \right )dx dy\, ,
\end{align}
where $\varepsilon(\rho)$ a function which describes the internal energy of the fluid. 
The stochastic process for this Lagrangian can be  directly computed from the coadjoint action and is 
\begin{align}
    \begin{split}
    du &=  u\cdot \nabla u\, dt + \sum_k((u\cdot \nabla)\sigma_k + \sigma_k (\nabla u)^T)\circ dW_t^k 
    - \frac{1}{\rho} \nabla p\,dt\\
    d\rho &= \mathrm{div}(\rho u) dt + \sum_k\mathrm{div}(\rho \sigma_k) \circ dW_t^k\, ,
    \end{split}
\end{align}
where $p:=\rho^2 \varepsilon'(\rho)$ and the divergence free $\sigma_i(u)$ only depend on $u$ .

In the deterministic case, one of the stable flows for the Euler's equation is  the shear flow solution, which has the form $u_e= (u(y),0)$ and a constant density $\rho_e= 1$, can be stable provided $u(y)$ does not have inflexion points for example. 
We will consider such equilibrium solution from now on, and the compatible noise can be seen to be of the form $\sigma_k = (\eta_k(y), 0)^T$ for a certain number of functions $\eta_k(y)$. 
This renders the adjoint and coadjoint actions of the stochastic fields
\begin{align}
     \mathrm{ad}_{(\sigma_k, 0)} (v,\rho)&= (\eta_k \partial_1 v_i e_i- v_2e_1 \eta', \eta_k \partial_1\rho ) \\
    \mathrm{ad}^*_{(\sigma_k,0)} (\delta u ,\delta \rho)&= - ( e_2\delta u_1  \eta_k' + e_i\eta_k \partial_1 \delta u_i  , \eta_k\partial_1 ( \delta \rho))\, ,
\end{align}
where $\eta_k'= \partial_2\eta_k(y)$ and $(e_1,e_2)$ the orthogonal basis of $\mathbb R^2$. To compute the transient stability,  recall that we need to evaluate the following quantity 
\begin{align}  
     A_h + B_h+ A_C+B_C  = \sum_i \langle D^2H_C\mathrm{ad}^*_{\sigma_i} \delta \mu_t  +  \mathrm{ad}_{\sigma_i} D^2H_C \delta\mu_t, \mathrm{ad}^*_{\sigma_i} \delta \mu_t\rangle\, ,
\end{align}
where $A$ and $B$ are the two terms of the right hand side, and the subscripts indicated whether the Hamiltonian or the Casimir is taken for the second variation. 
We begin with the Hamiltonian terms, i.e. $A_h$ and $B_h$, but first we compute the second variation of the Hamiltonian 
\begin{align*}
    \begin{split}
    \left \langle (\delta u_1, \delta \rho_1), D^2h\, (\delta u_2 , \delta \rho_2)\right \rangle &= \int  \delta \rho_1 ( \delta u_2 \cdot u_e ) + \delta \rho_2 (\delta u_1 \cdot u_e ) \\
    & + \epsilon''(1) \delta \rho_1 \delta \rho_2 + \delta u_1 \cdot \delta u_2 \, dx dy\, ,
    \end{split}
\end{align*}
where recall that $D^2h= D^2h(\mu_e,\rho_e)$ and $\rho_e=1$. In shear flow case, it becomes
\begin{align*}
    \begin{split}
        \langle (\delta u_1, \delta \rho_1), D^2h\, (\delta u_2 , \delta \rho_2)\rangle &= \int   \left (  \delta \rho_1 \delta u_{2,1} u+ \delta \rho_2 \delta u_{1,1} u \right .  \\
        &\left . + \epsilon''(1) \delta \rho_1 \delta \rho_2 + \delta u_1 \cdot \delta u_2 \,\right )  dx dy\, ,
    \end{split}
\end{align*}
where $\delta u_{1,1}$ is the first component of $\delta u_1$. 
We compute the first term and obtain
\begin{align}
    \begin{split}
        A_h  &=\langle \mathrm{ad}^*_{(\sigma_k,0)}(\delta u, \delta \rho), D^2h\, \mathrm{ad}^*_{(\sigma_k,0)} (\delta u , \delta \rho)\rangle \\
        &=\int  \Big( u\eta_k^2 \partial_1 \delta \rho \partial_1 \delta u_1   + \epsilon''(1) \eta_k^2  (\partial_1 \delta \rho)^2 \\ 
        &\hspace{10mm}+\delta u_1^2  \eta_k'^2 + \eta_k^2 |\partial_1 \delta u|^2+ 2\eta_k \eta_k' \delta u_1 \partial_1\delta u_2\Big ) \,  dx dy\, , 
    \end{split}
    \label{first}
\end{align}
where a sum over $k$ is understood throughout these computations. 
For the second term, we start with 
\begin{align*}
    D^2h\, (\delta u, \delta \rho) = (\delta \rho ue_1 + e_i \delta u_i, \delta u_1 u + \epsilon''(1) \delta \rho)\, , 
\end{align*}
with that the the adjoint action is 
\begin{align*}
    \mathrm{ad}_{(\sigma,0)} D^2h\, (\delta u, \delta \rho) &=  \left (e_1\eta_k u \partial_1 \delta \rho  +e_i \eta_k  \partial_1 \delta u_i - e_1 \delta u_2 \eta_k' ,  \eta_k u\partial_1\delta u_1  +\eta_k \epsilon''(1) \partial_1\delta \rho\right )\, . 
\end{align*}
We thus obtain
\begin{align}
    \begin{split}
    B_h &= \langle \mathrm{ad}_{(\sigma,0)} D^2h\, (\delta u, \delta \rho), \mathrm{ad}^*_{(\sigma,0)} (\delta u, \delta \rho) \rangle\\
    &= -\int \Big ( \eta_k' \eta_k \delta u_1  \partial_1 \delta u_2 - \eta_k'\eta_k \partial_1 \delta u_1 \delta u_2 \\
      &\hspace{18mm}+\eta_k^2 u\partial_1 \delta u_1 \partial_1 \delta \rho  + \eta_k^2 |\partial_1 \delta u_i|^2 \\
    &\hspace{18mm}+ \eta_k^2 u \partial_1  \delta \rho \partial_1\delta u_1  +\eta_k^2 \epsilon''(1) (\partial_1  \delta \rho)^2\Big )dxdy \, . 
    \end{split}
\end{align}
After integrating the second term by parts, we end up with only 
\begin{align}
    \begin{split}
    A_h+B_h&= \int \delta u_1^2(\eta_k')^2 dx dy \leq \Sigma_1^2 \int ||\delta u||^2\, ,
    \end{split}
\end{align}
where $\Sigma_1 > 0 \in \mathbb{R}$  is the smallest number such that
\begin{align}\label{eta-bound}
    \sum_k (\eta'_k)^2(y) \leq \Sigma_1^2, \quad \forall y\,.
\end{align}
We now need to compute $A_C$ and $B_C$. 
First, recall the Casimir of the compressible 2D Euler equation
\begin{align}
    C(u,\rho) = \int \rho\, \Phi\left (\frac{\omega}{ \rho}\right ) dx dy\, ,
\end{align}
where $\omega = e_3 \cdot \nabla \times u$ is the scalar vorticity and $\Phi$ is an arbitrary function.
Following \cite{holm1985nonlinear} and using the Bernoulli function $K$, the second variation of the Casimir can be written as  
\begin{align}
    \langle (\delta u, \delta \rho), D^2C\, (\delta u, \delta \rho)\rangle  = \int \frac{1}{\omega_e} K'(\omega_e) \left (\delta\left (\frac{\omega }{\rho}\right )\right )^2\, ,
\end{align}
and the shear flow we have $K'(\omega_e)= -u'(y)\frac{u(y)}{u''(y)}$.
We then use the equivalent form of the stochastic Euler equation for the vorticity $\omega$
\begin{align}\label{euler-vorticity}
    d\left ( \frac{\omega}{\rho}\right) = u\cdot \nabla\left (\frac{ \omega}{\rho}\right)\, + \sigma_i \cdot \nabla \left (\frac{\omega}{\rho}\right ) \circ dW_t^i\, . 
\end{align} 
We then directly compute as $A_C+B_C = D_{(\delta u, \delta \rho)}( (\delta u,\delta \rho), D^2C(\delta u,\delta \rho)$ to get 
\begin{align*}
    A_C+ B_C &= \sum_k\int \left ( \left (\partial_1 \delta\left (\frac{\omega }{\rho}\right )\right )^2+  \partial_1^2 \delta\left (\frac{\omega }{\rho}\right ) \delta\left (\frac{\omega }{\rho}\right )\right ) \frac{K'(\omega_e)}{\omega_e} \eta_k^2\, dx dy \\
    &=0 \, , 
\end{align*}
where we used integration by parts the fact that $\omega_e$ and $\eta$ are independent of $x$.  Finally, upon observing that the deterministic norm is positive or negative definite at $(u_e,1)$, we obtain the transient stochastic stability result by applying the theorem \ref{SEC}. 
In this case, $\Sigma_1^2 = \Sigma^2$ is proportional to the spatial derivative of the noise amplitude given in \eqref{eta-bound}. 

Finally, the incompressible case is a direct consequence of the compressible case as one can directly check that the computation follows similarly if $\rho=1$ and $\delta \rho=0$. 
In particular, the same result holds with $\Sigma_1^2 $ in equation \eqref{eta-bound}.

\begin{remark}
	For incompressible 2D Euler flows, the existence of the constant $\Sigma_1$ in equation \eqref{eta-bound} is necessary for the existence and uniqueness of the stochastic equation.  
	As shown in \cite{brzezniak2016existence}, the following condition on the fields $\sigma_k$ should hold 
	\begin{align*}
		\sup \sum_k |\sigma_k(x)|^2 + \mathrm{max}\left (\sum_k |\nabla \sigma_k(x)|^2 \right )< +\infty\, . 
	\end{align*}
	In our case the second term is bounded by $\Sigma_1^2$. 
\end{remark}

\section{Conclusion and open problems}

In this work, we have considered the problem of assessing the stability in probability of a certain class of equilibria of stochastic Lie-Poisson systems. 
These equilibria are critical points of the deterministic extended Hamiltonian and are equilibria of a stochastic Lie-Poisson dynamical system. 
We showed that the energy-Casimir method can be extended in the stochastic context to investigate the stability properties of these equilibria. 
An important condition for the method is that the equilibrium must remain time-independent upon adding the stochastic perturbation. 
This compels us to choose a particular form of noise compatible with the equilibrium solution of the deterministic system. 
This choice is rather restrictive but it seems to be required for such a result to hold.
By extending the energy-Casimir method, we have obtained additional information about the effects of noise on these equilibria.
In particular, the stability is only in a probabilistic sense and holds for a finite time only. 
The period of time where the stability in probability holds can be estimated from above and is directly related to the strength of the noise. 
Remarkably, in the fluid example, this time-scale only depends on the spatial derivative of the noise fields. 
This work suggests that further analysis can be done to obtain a general formula for the estimation of this time scale, or the constant $\Sigma$. 
Here, we have only assumed in our assumption \ref{assumption} that such a constant can be found in general and that we can compute it for the examples treated here. 
The present energy-Casimir approach for obtaining probabilistic stability conditions for stochastic Lie-Poisson Hamiltonian systems could possibly be extended to the energy-momentum method of \cite{simo1991stability}. However, we leave this extension for future work.
Other research directions may include the study of stability when geometric dissipation is used (see \cite{arnaudon2016noise}), or when feedback control is used, as in \cite{ bloch1990stabilization}, to make the rigid body's intermediate axis nonlinearly stable. 
{\small
\subsection*{Acknowledgements}

We thank Wei Pan for very helpful discussions during the course of this work. 
AA acknowledges partial support from an Imperial College London Roth Award. 
All the authors are also partially supported by the European Research Council Advanced Grant 267382 FCCA and the UK EPSRC Grant  EP/N023781/1 held by DH.
}

\bibliographystyle{alpha}
\bibliography{biblio.bib}

\begin{thebibliography}{HMRW85}

\bibitem[ACH16]{arnaudon2016noise}
Alexis Arnaudon, Alex~L Castro, and Darryl~D Holm.
\newblock Noise and dissipation on coadjoint orbits.
\newblock {\em arXiv preprint arXiv:1601.02249}, 2016.

\bibitem[Arn65]{arnold1965conditions}
Vladimir Arnold.
\newblock Conditions for non-linear stability of plane steady curvilinear flows
  of an ideal fluid.
\newblock {\em Doklady Akademii Nauk SSSR}, 162(5):773--777, 1965.

\bibitem[Arn66a]{arnold1966geometrie}
Vladimir Arnold.
\newblock Sur la g{\'e}om{\'e}trie diff{\'e}rentielle des groupes de lie de
  dimension infinie et ses applications {\`a} l'hydrodynamique des fluides
  parfaits.
\newblock {\em Annales de l'institut Fourier}, 16(1):319--361, 1966.

\bibitem[Arn66b]{arnold1966priori}
Vladimir~Igorevich Arnol'd.
\newblock An a priori estimate in the theory of hydrodynamic stability.
\newblock {\em Izvestiya Vysshikh Uchebnykh Zavedenii. Matematika}, (5):3--5,
  1966.

\bibitem[BFM16]{brzezniak2016existence}
Zdzis{\l}aw Brze{\'z}niak, Franco Flandoli, and Mario Maurelli.
\newblock Existence and uniqueness for stochastic 2d euler flows with bounded
  vorticity.
\newblock {\em Archive for Rational Mechanics and Analysis}, 221(1):107--142,
  2016.

\bibitem[Bis82]{bismut1982mecanique}
J.-M. Bismut.
\newblock M\'ecanique al\'eatoire.
\newblock In {\em Tenth {S}aint {F}lour {P}robability {S}ummer {S}chool---1980
  ({S}aint {F}lour, 1980)}, volume 929 of {\em Lecture Notes in Math.}, pages
  1--100. Springer, Berlin-New York, 1982.

\bibitem[Blo15]{bloch2003nonholonomic}
Anthony Bloch.
\newblock {\em Nonholonomic mechanics and control; 2nd ed.}
\newblock Interdisciplinary Applied Mathematics. Springer, New York, NY, 2015.

\bibitem[BM90]{bloch1990stabilization}
Anthony~M Bloch and Jerrold~E Marsden.
\newblock Stabilization of rigid body dynamics by the energy-casimir method.
\newblock {\em Systems \& Control Letters}, 14(4):341--346, 1990.

\bibitem[CHR16]{cruzeiro2016momentum}
Ana~Bela Cruzeiro, Darryl~D Holm, and Tudor~S Ratiu.
\newblock Momentum maps and stochastic clebsch action principles.
\newblock {\em arXiv preprint arXiv:1604.04554}, 2016.

\bibitem[EM70]{ebin1970groups}
David~G Ebin and Jerrold Marsden.
\newblock Groups of diffeomorphisms and the motion of an incompressible fluid.
\newblock {\em Annals of Mathematics}, pages 102--163, 1970.

\bibitem[GBH17]{gaybalmaz2016geometric}
Fran{\c c}ois Gay-Balmaz and Darryl~D Holm.
\newblock Variational principles for stochastic geophysical fluid dynamics.
\newblock {\em In preparation}, 2017.

\bibitem[HMR98]{holm1998euler}
Darryl~D Holm, Jerrold~E Marsden, and Tudor~S Ratiu.
\newblock The {E}uler--{P}oincar{\'e} equations and semidirect products with
  applications to continuum theories.
\newblock {\em Advances in Mathematics}, 137(1):1 -- 81, 1998.

\bibitem[HMRW85]{holm1985nonlinear}
Darryl~D. Holm, Jerrold~E. Marsden, Tudor Ratiu, and Alan Weinstein.
\newblock Nonlinear stability of fluid and plasma equilibria.
\newblock {\em Physics Reports}, 123(1):1 -- 116, 1985.

\bibitem[Hol08]{holm2008geometric}
Darryl~D. Holm.
\newblock {\em Geometric mechanics. {P}art {II}}.
\newblock Imperial College Press, London; distributed by World Scientific
  Publishing Co. Pte. Ltd., Hackensack, NJ, 2008.
\newblock Rotating, Translating and Rolling.

\bibitem[Hol15]{holm2015variational}
Darryl~D. Holm.
\newblock Variational principles for stochastic fluid dynamics.
\newblock {\em Proceedings of the Royal Society of London A: Mathematical,
  Physical and Engineering Sciences}, 471(2176), 2015.

\bibitem[Kha11]{khasminskii2011stochastic}
Rafail Khasminskii.
\newblock {\em Stochastic stability of differential equations}, volume~66.
\newblock Springer Science \& Business Media, 2011.

\bibitem[KP92]{kloeden1992numerical}
P.E. Kloeden and E.~Platen.
\newblock {\em Numerical Solution of Stochastic Differential Equations}.
\newblock Applications of Mathematics. Springer-Verlag, 1992.

\bibitem[LCO08]{lazaro2008stochastic}
Joan-Andreu L{\'a}zaro-Cam{\'{\i}} and Juan-Pablo Ortega.
\newblock Stochastic {H}amiltonian dynamical systems.
\newblock {\em Rep. Math. Phys.}, 61(1):65--122, 2008.

\bibitem[LRSM92]{lewis1992heavy}
Debra Lewis, T~Ratiu, JC~Simo, and Jerrold~E Marsden.
\newblock The heavy top: a geometric treatment.
\newblock {\em Nonlinearity}, 5(1):1, 1992.

\bibitem[MR99]{marsden1999book}
Jerrold~E. Marsden and Tudor~S. Ratiu.
\newblock {\em Introduction to mechanics and symmetry}, volume~17 of {\em Texts
  in Applied Mathematics}.
\newblock Springer-Verlag, New York, second edition, 1999.
\newblock A basic exposition of classical mechanical systems.

\bibitem[{\O}ks03]{oksendal2003stochastic}
Bernt {\O}ksendal.
\newblock {\em Stochastic differential equations}.
\newblock Springer, 2003.

\bibitem[SLM91]{simo1991stability}
Juan~C Simo, Debra Lewis, and Jerrold~E Marsden.
\newblock Stability of relative equilibria. part i: The reduced energy-momentum
  method.
\newblock {\em Archive for Rational Mechanics and Analysis}, 115(1):15--59,
  1991.

\bibitem[YM06]{Yoshimura2006209}
Hiroaki Yoshimura and Jerrold~E. Marsden.
\newblock Dirac structures in lagrangian mechanics part ii: Variational
  structures.
\newblock {\em Journal of Geometry and Physics}, 57(1):209 -- 250, 2006.

\bibitem[YM07]{Yoshimura2007381}
Hiroaki Yoshimura and Jerrold~E. Marsden.
\newblock Reduction of dirac structures and the hamilton-pontryagin principle.
\newblock {\em Reports on Mathematical Physics}, 60(3):381 -- 426, 2007.

\end{thebibliography}

\end{document}